\documentclass[a4paper,12pt]{article}
\usepackage{amsmath,amssymb,enumerate,amsthm}

\newcommand{\R}{\mathbb{R}}

\newcommand{\N}{\mathbb{N}}
\newcommand{\ep}{\varepsilon}
\newcommand{\pa}{\partial}

\DeclareMathOperator{\supp}{supp}
\newcommand{\lr}[1]{{}\langle{}#1{}\rangle{}}

\newtheorem{theorem}{Theorem}[section]
\newtheorem{lemma}[theorem]{Lemma}
\newtheorem{proposition}[theorem]{Proposition}
\newtheorem{corollary}[theorem]{Corollary}
\theoremstyle{remark}
\newtheorem{remark}{Remark}[section]
\theoremstyle{definition}

\newtheorem{definition}{Definition}[section]

\numberwithin{equation}{section}

\makeatletter
\def\@cite#1#2{[{{\bfseries #1}\if@tempswa , #2\fi}]}
\makeatother                  

\begin{document}
\begin{center}
\Large{{\bf
Lifespan estimates 
for semilinear damped wave equation 
in a two-dimensional exterior domain
}}
\end{center}

\vspace{5pt}

\begin{center}
Masahiro Ikeda%
\footnote{Center for Advanced Intelligence Project,
RIKEN,
Nihonbashi 1-chome Mitsui Building, 15th floor, 1-4-1 Nihonbashi,Chuo-ku, Tokyo, 103-0027, Japan, 
E-mail:\ {\tt masahiro.ikeda@riken.jp}},
Motohiro Sobajima%
\footnote{
Department of Mathematics, 
Faculty of Science and Technology, Tokyo University of Science,  
2641 Yamazaki, Noda-shi, Chiba, 278-8510, Japan,  
E-mail:\ {\tt msobajima1984@gmail.com}},
Koichi Taniguchi%
\footnote{
Advanced Institute for Materials Research,
Tohoku University,
2-1-1 Katahira, Aoba-ku, Sendai, 980-8577, Japan,
E-mail:\ {\tt koichi.taniguchi.b7@tohoku.ac.jp}}
\ and\ 
Yuta Wakasugi%
\footnote{
Laboratory of Mathematics, Graduate 
School of Advanced Science and 
Engineering, Hiroshima University, 
Higashi-Hiroshima, 739-8527, Japan, 
E-mail:\ {\tt wakasugi@hiroshima-u.ac.jp}}
\end{center}

\newenvironment{summary}{\vspace{.5\baselineskip}\begin{list}{}{%
     \setlength{\baselineskip}{0.85\baselineskip}
     \setlength{\topsep}{0pt}
     \setlength{\leftmargin}{12mm}
     \setlength{\rightmargin}{12mm}
     \setlength{\listparindent}{0mm}
     \setlength{\itemindent}{\listparindent}
     \setlength{\parsep}{0pt}
     \item\relax}}{\end{list}\vspace{.5\baselineskip}}
\begin{summary}
{\footnotesize {\bf Abstract.}
Lifespan estimates for semilinear damped wave equations of the form $\pa_t^2u-\Delta u+\pa_tu=|u|^p$ in a two dimensional exterior domain endowed with the Dirichlet boundary condition are dealt with. 
For the critical case of the semilinear heat equation $\pa_tv-\Delta v=v^2$ with the Dirichlet boundary condition and the initial condition $v(0)=\ep f$,  the corresponding lifespan can be estimated from below and above by $\exp(\exp(C\ep^{-1}))$ with different constants $C$. 
This paper clarifies that the same estimates hold even for the critical semilinear damped wave equation in the exterior of the unit ball under the restriction of radial symmetry.
To achieve this result, a new technique to control $L^1$-type norm and a new Gagliardo--Nirenberg type estimate with logarithmic weight are introduced.
}
\end{summary}



\tableofcontents

\newpage
\section{Introduction}
In this paper, we consider the initial-boundary value
problem of the semilinear damped wave equation 
in the exterior of the two-dimensional 
closed
unit ball $B=\{x\in \R^2\;;\;|x|\leq 1\}$, that is, 
\begin{equation}\label{intro:problem}
\begin{cases}
\pa_t^2u(x,t)-\Delta u(x,t)+\pa_tu(x,t)=|u(x,t)|^p
&\text{in}\ B^c\times (0,T), 
\\
u(x,t)=0
&\text{on}\ \pa B^c\times (0,T), 
\\
(u,\pa_tu)(x,0)=(0,\ep g(x)), &\text{in}\ B^c.
\end{cases}
\end{equation}
Here $p>1$ indicates the structure of the nonlinear term. 
The function $g:B^c\to \R$ is given
and $u:B^c\times[0,T)\to \R$ is unknown. 
The constant $\ep>0$ is the parameter describing  
the smallness of initial data. 
Our interest is 
the behavior of solutions to 
\eqref{intro:problem} with 
small initial data.

For the semilinear heat equation
\begin{equation}\label{intro:heat-whole}
\begin{cases}
\pa_tv(x,t)-\Delta v(x,t)=v(x,t)^p
&\text{in}\ \R^N\times (0,T), 
\\
v(x,0)=f(x)\geq 0, &\text{in}\ \R^N,
\end{cases}
\end{equation}
from the pioneering work \cite{Fujita1966} by Fujita, 
there are many papers dealing with 
the existence/nonexistence 
of 
global solutions to \eqref{intro:heat-whole} (see Quittner--Souplet \cite{QSbook}). 
Nowadays, 
the exponent $p_F(N)=1+\frac{2}{N}$ is well-known 
as the threshold for dividing the situation 
of the existence/nonexistence of nonnegative global solutions, 
namely, 
\begin{itemize}
\item if $1<p\leq p_F(N)$, 
then \eqref{intro:heat-whole} does not 
possess non-trivial global solutions;
\item if $p>p_F(N)$, 
then \eqref{intro:heat-whole} possesses a non-trivial 
global solution.
\end{itemize}
The similar phenomenon occurs 
also for 
the semilinear damped wave equation 
\begin{equation}\label{intro:dw-whole}
\begin{cases}
\pa_t^2u(x,t)-\Delta u(x,t)+\pa_tu(x,t)=|u(x,t)|^p
&\text{in}\ \R^N\times (0,T), 
\\
(u,\pa_tu)(x,0)=(u_0(x),u_1(x)), &\text{in}\ \R^N.
\end{cases}
\end{equation}
The pioneering work for the problem \eqref{intro:dw-whole} is the paper 
\cite{Matsumura1976} 
by Matsumura 
via the analysis of the profile of linear solutions 
in the following inequalities
\begin{gather}\label{intro:matsumura1}
\|\pa_t^k \pa_x^\alpha u(t)\|_{L^\infty(\R^N)}
\leq C(1+t)^{-\frac{N}{2}-k-\frac{|\alpha|}{2}}
(\|u_0\|_{H^{m+1}\cap L^1(\R^N)}
+\|u_1\|_{H^{m}\cap L^1(\R^N)})
, 
\\
\label{intro:matsumura2}
\|\pa_t^k \pa_x^\alpha u(t)\|_{L^2(\R^N)}
\leq C(1+t)^{-\frac{N}{4}-k-\frac{|\alpha|}{2}}
(\|u_0\|_{H^{\widetilde{m}+1}\cap L^1(\R^N)}
+\|u_1\|_{H^{\widetilde{m}}\cap L^1(\R^N)})
\end{gather}
(with $m=[\frac{N}{2}]+k+|\alpha|$ and $\widetilde{m}=k+|\alpha|-1$), 
which are   
so-called ``Matsumura estimates''. 
Until Todorova--Yordanov \cite{ToYo2001}
and Zhang \cite{Zhang2001},
(a rough description of) the situation of 
the existence/nonexistence of 
non-trivial (small) global solutions 
to \eqref{intro:dw-whole} 
is clarified as follows: 
\begin{itemize}
\item if $1<p\leq p_F(N)$ and 
$\int_{\R^N}(u_0+u_1)\,dx>0$, 
then \eqref{intro:dw-whole} does not 
possess non-trivial global solutions
(a kind of smallness does not provide global solutions);
\item if $p>p_F(N)$, 
then \eqref{intro:heat-whole} possesses a non-trivial 
global solution.
\end{itemize}
After that, the precise estimates of 
the lifespan $T_\ep$ (maximal existence time) 
of solutions to \eqref{intro:dw-whole} 
with the small initial data $(\ep f,\ep g)$
became the subject of interest. 
It is firstly 
discussed in Li--Zhou \cite{LZ1995}, 
and until Lai--Zhou \cite{LZ2019}
the sharp lifespan estimates are clarified as the following:
for sufficiently small $\ep>0$, 
\begin{align}
\begin{cases}
c_p\ep^{-(\frac{1}{p-1}-\frac{N}{2})^{-1}}
\leq T_\ep \leq  
C_p\ep^{-(\frac{1}{p-1}-\frac{N}{2})^{-1}}
&\text{if}\ 1<p<p_F(N), 
\\
\exp(c_{p}\ep^{-(p-1)})\leq T_\ep \leq \exp(C_p\ep^{-(p-1)})
&\text{if}\ p=p_F(N)
\end{cases}
\end{align}
for some positive constants $c_p$ and $C_p$ (independent of $\ep$);
note that these estimates are completely same as  
the ones for the blowup time of solutions to 
the semilinear heat equation \eqref{intro:heat-whole}
with the initial condition $v(x,0)=\ep f(x)$ (having the small parameter $\ep$).

In the case of 
the problem of the semilinar heat equation 
in an $N$-dimensional exterior domain $\Omega$
with the Dirichlet boundary condition:
\begin{equation}\label{intro:heat-exterior}
\begin{cases}
\pa_tv(x,t)-\Delta v(x,t)=v(x,t)^p
&\text{in}\ \Omega\times (0,T), 
\\
v(x,t)=0
&\text{on}\ \Omega\times (0,T), 
\\
v(x,0)=f(x)\geq 0, &\text{in}\ \Omega, 
\end{cases}
\end{equation}
the situation is almost similar as in the case of the whole space
when $N\geq 3$. 
Actually, for the linear case, 
Grigor'yan and Saloff-Coste \cite{GS2002} discussed 
the asymptotic behavior of the Dirichlet heat kernel 
for the exterior of a compact set in Riemannian manifolds. 
It is shown that 
the Dirichlet heat kernel for an $N$-dimensional exterior domain $(N\geq 3)$
behaves like the one for $\R^N$ in the far field.  
In contrast, the Dirichlet heat kernel 
in two-dimensional exterior domains cannot be approximated 
by the one for $\R^2$. 
This fact can be explained as the transient (recurrence) property for $N\geq 3$ ($N=2$) 
of the Brownian motion in $\R^N$.  
This significant difference reflects the difficulty 
of the case of two-dimensional exterior problem. 
By using the Kaplan's method via the Dirichlet heat kernel found in \cite{GS2002}, 
Pinsky \cite{Pinsky2009} tried to draw the complete picture 
of the existence/nonexistence of global solutions to \eqref{intro:heat-exterior}, 
however, 
the critical case $p=p_F(2)=2$ seems 
to have a gap in his proof.
Later, nonexistence of global solutions for the case $p=2$ 
is proved in Ikeda--Sobajima \cite{IkSo2019JMAA} via a sharpened test function method
with a distinctive shape of the (double exponential type) lifespan estimate 
\begin{equation}
\label{intro:double}
T_\ep \leq \exp(\exp(C\ep^{-1}))
\end{equation}
for the solution with the initial condition $v(x,0)=\ep f(x)$.  
One can prove that this is actually the sharp lifespan estimate 
via the supersolution-subsolution method (explained in \cite[Section 20]{QSbook})
with a supersolution
\[
U(x,t)= 
\alpha(t)
e^{t\Delta_\Omega}f, \quad 
\alpha(t)
=\ep \left(1-(p-1)\ep^{p-1}\int_0^t\|e^{s\Delta_\Omega}f\|_{L^\infty(\Omega)}^{p-1}\,ds\right)^{-\frac{1}{p-1}}
\]
with  the $L^\infty$-estimate of the Dirichlet heat semigroup 
$e^{t\Delta_\Omega}:$
\begin{equation}
\label{intro:eq:Linfty-logL1}
\|e^{t\Delta_\Omega}f\|_{L^\infty(\Omega)}\leq Ch(t)\big(\|(\log |x|) f\|_{L^1(\Omega)}+\|f\|_{L^\infty(\Omega)}\big), \quad t>0
\end{equation}
with the decay rate involving the logarithmic function
\begin{equation}\label{intro:h}
h(t)=\frac{1}{(1+t)(1+\log(1+t))}.
\end{equation}
This is valid only for the case of two-dimensional (general) exterior domains.

For the semilinear damped wave equation \eqref{intro:problem}, 
the existence/nonexistence of solutions to \eqref{intro:problem}
also has been dealt with in the literature 
(e.g., Ikehata \cite{Ikehata2005JMAA} for the existence
and Ogawa--Takeda \cite{OT2009} for the nonexistence).
Although
the effect of the recurrence of the Brownian motion in $\R^2$ could appear also in the analysis of the damped wave equation, however,
studies from such a viewpoint are few. 
Only in \cite{IkSo2019JMAA}, one can find 
that 
the lifespan estimate of the solution $u$ to 
the semilinear damped wave equation \eqref{intro:problem} 
has the same upper bound (as in \eqref{intro:double})
as the case of the semilinear heat equation.
In this connection, 
the question about the sharpness 
of this lifespan estimate naturally arises.
The purpose of the present paper 
is to address this problem, 
that is, 
to clarify the sharpness of 
the (double exponential type) lifespan estimate 
for the two-dimensional exterior problem of the
semilinear damped wave equation \eqref{intro:problem}.

To state the result, 
we clarify the definition of solutions to \eqref{intro:problem} 
as follows. 
\begin{definition}\label{def:sol}
For general open set $\Omega$ in $\R^N$ (with a smooth boundary), we denote 
$\Delta_{\Omega}$ as the Laplacian 
endowed with the domain $D(\Delta_\Omega)=H^2(\Omega)\cap H_0^1(\Omega)$ 
(which describes the Laplace operator $\Delta$ involving the Dirichlet boundary condition).  
Then we define 
that 
$u:B^c\times [0,T)\to \R$ 
is a weak solution of \eqref{intro:problem} 
in $(0,T)$ with the initial condition $(u,\pa_tu)(0)=(u_0,u_1)\in H_0^1(B^c)\times L^2(B^c)$ 
if  
$u\in C^1([0,T);L^2(B^c))\cap C([0,T);H_0^1(B^c))$ and 
\[
\big(u(t),\pa_tu(t)\big)
=
e^{t\mathcal{L}_{B^c}}(u_0,u_1)
+
\int_{0}^t e^{(t-s)\mathcal{L}_{B^c}}(0,|u(s)|^p)\,ds,
\quad t\in (0,T),
\]
where $(e^{t\mathcal{L}_{B^c}})_{t\geq0}$ 
is the $C_0$-semigroup on 
$\mathcal{H}=H_0^1(B^c)\times L^2(B^c)$ 
generated by $\mathcal{L}_{B^c}(u,v)=(v,\Delta_{B^c} u-v)$ 
with the domain $D(\mathcal{L}_{B^c})=
D(\Delta_{B^c})\times H_0^1(B^c)$.
\end{definition}
\begin{remark}\label{rem:S(t)}
It turns out that the following representation 
of solution $u$ of \eqref{intro:problem} with 
initial condition $(u,\pa_tu)(0)=(u_0,u_1)$ 
is also valid: 
\[
u(t)=
\pa_t S(t)u_0+S(t)(u_0+u_1)+\int_{0}^t S(t-s)|u(s)|^p\,ds,
\]
where $S(t)g:=P_1e^{t\mathcal{L}_{B^c}}(0,g)$ 
with 
the projection $P_1(u,v)=u$. Our argument in the present paper 
relies on this representation.  
\end{remark}

Existence and uniqueness of solutions to \eqref{intro:problem}
and also the blowup alternative are well-known (see e.g., 
Ikehata \cite{Ikehata2005JMAA}). 
\begin{proposition}\label{prop:ndw:fundamental}
The following assertions hold:
\begin{itemize}
\item[\bf (i)]
For every $(u_0,u_1)\in H_0^1(B^c)\times L^2(B^c)$, 
there exist a positive constant $T$ 
and $u\in C^1([0,T);L^2(B^c))\cap C([0,T);H_0^1(B^c))$ 
such that $u$ is a unique weak solution of \eqref{intro:problem} in $(0,T)$
with the initial condition $(u,\pa_tu)(0)=(u_0,u_1)$. 
\item[\bf (ii)] 
The weak solution $u$ of \eqref{intro:problem} in a bounded interval $(0,T)$ 
cannot be able to extend to a solution in a wider interval 
if and only if 
\begin{equation}
\lim_{t\to T}
\Big(\|\pa_tu(t)\|_{L^2(B^c)}+
\|\nabla u(t)\|_{L^2(B^c)}+
\|u(t)\|_{L^2(B^c)}\Big)=+\infty.
\end{equation}
\end{itemize}
\end{proposition}

By virtue of Proposition \ref{prop:ndw:fundamental}, 
we can define the lifespan of solutions to
\eqref{intro:problem}.

\begin{definition}\label{def:lifespan}
Define the lifespan $T_{\max}(u_0,u_1)\in (0,\infty]$ as 
the maximal existence time of 
weak solutions to \eqref{intro:problem} 
with the initial condition $(u,\pa_tu)(0)=(u_0,u_1)$. 
Namely, 
\[
T_{\max}(u_0,u_1)=\sup \big\{T>0\;;\;
\text{\eqref{intro:problem} has a weak solution in $(0,T)$}\big\}.
\]
\end{definition}

In Ikeda--Sobajima \cite{IkSo2019JMAA}, blowup 
of solutions to \eqref{intro:problem} with small initial data
is proved under the initial condition $(u,\pa_tu)(x,0)=(\ep f,\ep g)$ 
with 
\[
\int_{B^c} \big(f(x)+g(x)\big)\log |x|\,dx>0.
\]
We should point out that the weight function $\log |x|$ has been chosen as 
the positive harmonic function satisfying the Dirichlet boundary condition.
To reflect the above situation 
and also the $L^\infty$-decay estimate 
\eqref{intro:eq:Linfty-logL1}
for $e^{t\Delta_{B^c}}$ 
to our consideration in the present paper, 
we need to introduce the following 
$L^p$-spaces with weighted measure involving $\log |x|$. 
\begin{definition}
Define the measure $d\mu=(1+\log |x|)\,dx$ and  
for $1\leq p<\infty$, 
\[
L_{d\mu}^p=
\left\{f\in L^p(B^c)\;;\;
\|f\|_{L_{d\mu}^p}<+\infty
\right\},
\quad \|f\|_{L_{d\mu}^p}=\left(\int_{B^c}|f(x)|^p\,d\mu\right)^{\frac{1}{p}}.
\]
\end{definition}

Now we are in a position to state our result 
for the lower bound for the lifespan estimate
of solutions to \eqref{intro:problem}
under the radially symmetric setting.
The following assertion 
is formulated in the subspace of radially symmetric 
functions in $L^2(B^c)$: 
\[
L^2_{\rm rad}=\{f\in L^2(B^c)\;;\;\text{$f$ is radially symmetric}\}.
\]
\begin{theorem}\label{thm:main1}
If $g\in L_{\rm rad}^2\cap L_{d\mu}^1$, 
then 
the following assertions hold: 

\begin{itemize}
\item[\bf (i)]
If $1< p\leq 2$, then 
there exist positive constants 
$\ep_0>0$ 
 and $c>0$ 
such that for $\ep\in (0,\ep_0]$, 
\[
T_{\max}(0,\ep g)\geq 
\begin{cases}
c\left(\dfrac{1}{\ep}\log \dfrac{1}{\ep}\right)^{\frac{p-1}{2-p}} 
&\text{if}\ 1<p<2, 
\\[3pt]
\exp (\exp(c\ep^{-1}))
&\text{if}\ p=2.
\end{cases}
\]
\item[\bf (ii)]If $2<p<\infty$, then there 
exist positive constants 
$\delta$ and $C$ 
such that
if $\|g\|_{L^2(B^c)}+\|g\|_{L_{d\mu}^1}\leq \delta$, then 
$T_{\max}(0,g)=+\infty$ with 
\begin{align*}
\|u(t)\|_{L_{d\mu}^1}&\leq C\delta, \quad t>0, 
\\
\|\nabla u(t)\|_{L^2(B^c)}&\leq \frac{C\delta}{(1+t)(1+\log (1+t))}, \quad t>0, 
\\
\|\pa_t u(t)\|_{L^2(B^c)}&\leq \frac{C\delta}{(1+t)^{\frac{3}{2}}(1+\log (1+t))}, \quad t>0.
\end{align*}
\end{itemize}
\end{theorem}
The lower bounds of the lifespan in Theorem \ref{thm:main1} {\bf (i)} 
are derived from the following relation:
\[
c_p\leq \ep^{p-1}\int_0^{T_{\max}(0,\ep g)} h(t)^{p-1}\,dt
\]
for some small constant $c_p$, 
where $h(t)$ is given in \eqref{intro:h}.
Combining Theorem \ref{thm:main1} {\bf (i)} with 
the upper bound of the lifespan given in Ikeda--Sobajima \cite{IkSo2019JMAA}, 
we obtain the following:
\begin{corollary}\label{cor:1<pleq2}
Let $g\in L^2_{\rm rad}\cap L^1_{d\mu}$ and
let $T_{\max}$ be in Definition \ref{def:lifespan}.
If $g\geq 0$ and $g\not\equiv 0$, 
then 
for every $1<p\leq2$, one has  
\begin{align*}
0<\liminf_{\ep \to 0} \Big(\ep^{p-1}\int_0^{T_{\max}(0,\ep g)} h(t)^{p-1}\,dt\Big)
\leq 
\limsup_{\ep \to 0} \Big(\ep^{p-1}\int_0^{T_{\max}(0,\ep g)} h(t)^{p-1}\,dt\Big)
<+\infty.
\end{align*}
\end{corollary}
\begin{remark}
We do not know whether the quantity 
\[
\ep^{p-1}\int_0^{T_{\max}(0,\ep g)} h(t)^{p-1}\,dt
\]
converges to a constant as $\ep \to 0$ or not.
\end{remark}
The counter part about the global existence, 
we also have the assertion for the total energy decay.
\begin{corollary}\label{cor:p>2}
Let $2<p<\infty$ and
let $u$ be the unique global solution of \eqref{intro:problem}
obtained in Theorem \ref{thm:main1}. 
Then one has the energy decay estimate for $u$:
\[
\int_{B^c}\Big(|\nabla u(t)|^2+(\pa_tu(t))^2\Big)\,dx\leq 
\frac{C^2\delta^2}{(1+t)^2(1+\log(1+t))^2}, \quad t>0.
\]
\end{corollary}
\begin{remark}
Ono \cite{Ono2003} obtained the energy decay estimate 
\[
\int_{B^c}\Big(|\nabla u(t)|^2+(\pa_tu(t))^2\Big)\,dx\leq 
\frac{C_\delta}{(1+t)^{2-\delta}}, \quad t>0
\]
for the solution $u$ of linear damped wave equation 
with initial data $(u,\pa_tu)(0)=(u_0,u_1)\in [H_0^1\cap L^1(B^c)]\times [L^2\cap L^1(B^c)]$.
After that, 
 Ikehata \cite{Ikehata2005JMAA} removed 
the loss $\delta$ of the decay rate by assuming $|x|\log |x| (u_0+u_1)\in L^2(B^c)$. 
The decay estimate in Corollary \ref{cor:p>2} 
(even though it is the semilinear problem)
is faster than those by a different assumption $(u,\pa_tu)(0)=(0,g)$ 
with $g\in L^2_{\rm rad}\cap L^1_{d\mu}$. 
Incidentally, one can find that the decay of total energy 
in Corollary \ref{cor:p>2} 
is faster than the local energy decay proved in Dan--Shibata \cite{DS1995}. 
\end{remark}

Now we shall describe the strategy of the present paper. 
Here we only focus our attention to the two-dimensional case. 
If we move to the semilinear heat equation (of course in an exterior domain), 
we can employ the supersolution-subsolution method as explained before. 
However, it seems impossible to apply this argument for the hyperbolic equation \eqref{intro:problem}. 
Instead of this defect, 
in the case of 
the Cauchy problem of the semilinear damped wave equation \eqref{intro:dw-whole},
the positively preserving property of 
the linear solution map $S_*(t)=P_1e^{t\mathcal{L}_{\R^2}}$ 
seems useful to control the $L^1$-type norm of the solution $S_*(t)g$. 
Together with the basic energy functional of the form  
$\|\pa_tu\|_{L^2(\R^2)}^2+\|\nabla u\|_{L^2(\R^2)}^2$, 
this consideration enables us to control the $L^1$-norm of the semilinear term $|u|^p$ 
via the Gagliardo--Nirenberg inequality 
\[
\|f\|_{L^p(\R^2)}\leq C_p\|\nabla f\|_{L^2(\R^2)}^{1-\frac{1}{p}}\|f\|_{L^1(\R^2)}^{\frac{1}{p}} \quad f\in H^1(\R^2)\cap L^1(\R^2)
\]
(when $p=2$, this is called the Nash inequality).
We emphasize that the $L^1$-norm is a conserved quantity of the linear heat semigroup $e^{t\Delta_{\R^2}}$ (for nonnegative initial data)
which represents the asymptotic behavior of solutions \eqref{intro:heat-whole}.
Employing this procedure with the Matsumura estimate \eqref{intro:matsumura2}, 
we can reach the sharp lower bound of lifespan 
of the solution to \eqref{intro:dw-whole}  
with the initial condition 
$(u,\pa_tu)(0)=(0,\ep g)$ with $g\in L^2(\R^2)\cap L^1(\R^2)$. 
One of the novelty of the present paper 
is to introduce this kind of strategy in the analysis of the semilinear damped wave equation. 

Let us go back to the target problem \eqref{intro:problem}. 
In this case, a conserved quantity for the Dirichlet heat semigroup $e^{t\Delta_{B_c}}$ 
can be chosen as 
\[
I[f]=\int_{B_c}f(x)\log |x|\,dx, 
\]
where the weight $\log |x|$ is chosen as a positive harmonic function 
satisfying the Dirichlet boundary condition.
Then it turns out that the behavior of the above quantity 
in the problem \eqref{intro:problem} can be tracked by 
the ordinary differential equation 
\begin{align*}
\left(
\frac{d^2}{dt^2}
+
\frac{d}{dt}
\right)
\int_{B_c}u(x,t)\log |x|\,dx
=
\int_{B_c}|u(x,t)|^p\log |x|\,dx.
\end{align*}
This suggests that the following version of 
the Gagliardo--Nirenberg type inequality 
seems reasonable:
\[
\int_{B^c}|f|^p \log |x|\,dx\leq \widetilde{C}_p\|\nabla f\|_{L^2(B^c)}^{p-1}
\int_{B^c}|f| \log |x|\,dx, 
\quad f\in H^1_0(B^c)\cap L_{d\mu}^1,
\]
which is not discussed so far. 
To proceed this strategy, 
the positively preserving property of the linear solution map $S(t)$ is essential. 
To justify this property, additionally we assume the radial symmetry. 
Moreover, to reach the sharp lifespan estimate, 
we also use 
the corresponding Matsumura type estimate which 
should have a logarithmic decay factor as in \eqref{intro:eq:Linfty-logL1}. 
This could be proved via the diffusion phenomenon 
with the decay estimate via the estimate for the Dirichlet heat semigroup $e^{t\Delta_{B^c}}$.

The present paper is organized as follows. 
In section \ref{sec:preliminary}, 
we collect the important tools in this paper. 
More precisely, 
an abstract version of Matsumura estimates,  
decay estimates for the Dirichlet heat semigroup $e^{t\Delta_{B^c}}$
(an alternative proof is 
written in Appendix), 
the usual Gagliardo--Nirenberg estimates 
and the critical Hardy inequality are listed. 
In section \ref{sec:linear}, 
the linear solution $S(t)g$ is analysed. 
Here we prove 
a modified Matsumura estimate with decay involving logarithmic factor, 
the positively preserving property 
and an $L^1_{d\mu}$-estimate for $S(t)g$.  
Section \ref{sec:semilinear} is 
devoted to 
the proof of the lifespan estimate from below.
A Gagliardo--Nirenberg inequality with logarithmic weight 
is proved by 
the use of the critical Hardy inequality 
at the beginning of Section \ref{sec:semilinear}. 

\section{Preliminaries}\label{sec:preliminary}
\paragraph{}

In this section 
we collect several important tools 
to analyse the target problem \eqref{intro:problem}.

\subsection{Abstract version of Matsumura estimates}

We first state 
an abstract version of 
Matsumura estimates 
which is valid 
for the second order differential equation 
\begin{equation}\label{eq:dw-abst}
\begin{cases}
u''(t)+Au(t)+u'(t)=0, \quad t>0, 
\\
(u,u')(0)=(u_0,u_1)
\end{cases}
\end{equation}
in a Hilbert space $H$. Here $A$ is a nonnegative 
selfadjoint operator in $H$ 
endowed with domain $D(A)$. 
Existence and uniqueness of solutions to \eqref{eq:dw-abst} 
are verified via the well-known Hille--Yosida theorem.  
Here we denote by $S_A(t)g$  
the solution $u\in C^1([0,\infty);H)\cap C([0,\infty);D(A^{1/2}))$
of \eqref{eq:dw-abst} 
with the initial condition $(u,u')(0)=(0,g)$. 
Note that by using $S_A(t)$, one can find that 
the solution of \eqref{eq:dw-abst} 
has a representation
\begin{equation*}
u(t)=\frac{d}{dt}[S_A(t)u_0]+S_A(t)[u_0+u_1].
\end{equation*}
The following lemma
provides 
the asymptotic profile of $S_A(t)g$ (in the sense of the energy functional), 
which can be described by using the $C_0$-semigroup $(e^{-tA})_{t\geq 0}$ 
on $H$ generated by $-A$. 
\begin{lemma}[{Radu--Todorova--Yordanov \cite{RaToYo2011}}]\label{lem:matsumura-abst}
The following assertions hold:
\begin{itemize}
\item[\bf (i)]
There exists a positive constant $C_{{\rm M},1}$ such that 
for every $g\in H$ and $t\geq 1$, 
\begin{align*}
\big\|A^{1/2} (S_A(t)-e^{-tA})g\big\|_{H}
\leq 
C_{{\rm M},1}\Big(
t^{-\frac{3}{2}}\|e^{-\frac{t}{2}A}g\|_{H}
+e^{-\frac{t}{16}}\|(A^{1/2}+1)^{-1}A^{1/2}g\|_H
\Big).
\end{align*}
\item[\bf (ii)]
There exists a positive constant $C_{{\rm M},2}$ such that 
for every $g\in H$ and $t\geq 1$, 
\begin{align*}
\left\|\frac{d}{dt}(S_A(t)-e^{-tA})g\right\|_{H}
\leq 
C_{{\rm M},2}t^{-2}\Big(
\|g\|_{H}
+e^{-\frac{t}{4}}\|g\|_H
\Big).
\end{align*}
\end{itemize}
\end{lemma}
\begin{remark}
The assertion {\bf (ii)} in Lemma \ref{lem:matsumura-abst} 
is not written explicitly in \cite{RaToYo2011}, but  
the same strategy also provides the estimate for 
the derivative in $t$. 
\end{remark}
In our situation, 
we employ Lemma \ref{lem:matsumura-abst} 
with the (negative) Dirichlet Laplacian 
$-\Delta_{B^c}$ in the Hilbert space $L^2(B^c)$. 
The corresponding 
linear damped wave equation is as follows:
\begin{equation}\label{eq:dw-linear}
\begin{cases}
\pa_t^2u(x,t)-\Delta u(x,t)+\pa_tu(x,t)=0
&\text{in}\ B^c\times (0,\infty), 
\\
u(x,t)=0
&\text{on}\ \pa B^c\times (0,\infty), 
\\
(u,\pa_tu)(x,0)=(0,g(x)), &\text{in}\ B^c.
\end{cases}
\end{equation}
The classical energy identity for the solution 
$S(t)g=P_1e^{t\mathcal{L}_{B^c}}(0,g)$ 
(given in Remark \ref{rem:S(t)}) provides 
the following basic inequality.
\begin{lemma}\label{lem:energy-est}
For every $g\in L^2(B^c)$, one has
\[
\|\nabla S(t)g\|_{L^2(B^c)}^2+
\|\pa_t S(t)g\|_{L^2(B^c)}^2\leq \|g\|_{L^2(B^c)}^2, \quad t\geq 0.
\]
\end{lemma}
The following is the result by \cite{RaToYo2011} applied to 
the case \eqref{eq:dw-linear}.
\begin{lemma}\label{lem:matsumura-applied}
There exist positive constants $C_{{\rm M},1}'$ and $C_{{\rm M},2}'$ 
such that for every $g\in L^2(B^c)$ and $t\geq 1$, 
\begin{align*}
\|\nabla (S(t)-e^{t\Delta_{B^c}})g\|_{L^2(B^c)}
&\leq 
C_{{\rm M},1}'t^{-\frac{3}{2}}\|g\|_{L^2(B^c)},
\\
\|\pa_t(S(t)-e^{t\Delta_{B^c}})g\|_{L^2(B^c)}
&\leq 
C_{{\rm M},2}'t^{-2}\|g\|_{L^2(B^c)}.
\end{align*}
\end{lemma}

\subsection{Dirichlet heat semigroup on the exterior domain $B^c$}

Here we state the $L^p$-$L^q$ type 
estimates for $e^{t\Delta_{B^c}}$ 
with logarithmic weight. 
Although 
it can be obtained
via the heat kernel estimate 
in \cite{GS2002} 
(as explained in Introduction), 
in Appendix
we give an alternative proof 
based on the technique of the analysis of partial differential equations
of parabolic type. 

In the discussion of the present paper, the following 
estimates are crucial.

\begin{lemma}\label{lem:L^p-wL^1}
For every $q\in [1,2]$, 
there exists a positive constant $C_{{\rm H},q}>0$ 
such that 
if $f\in L_{d\mu}^q$, then
for every $t>0$, one has $e^{t\Delta_{B^c}}f\in L^2(B^c)$ with  
\[
\|e^{t\Delta_{B^c}}f\|_{L^2(B^c)}
\leq 
\frac{C_{{\rm H},q}}{t^{\frac{1}{q}-\frac{1}{2}}(1+\log (1+t)
)^{\frac{1}{q}}}\|f\|_{L_{d\mu}^q}. 
\]
\end{lemma}
\subsection{Some functional inequalities}

In the present paper, the following form of the Gagliardo--Nirenberg inequalities 
will be used (cf. Friedman \cite{FriedmanBook}); 
note that the following inequality with $q=2$ is also called the Nash inequality.
\begin{lemma}\label{lem:GN}
For every $1<q<\infty$, 
one has $H_0^1(B^c)\cap L^1(B^c) \subset L^q(B^c)$ 
and there exists a positive constant $C_{{\rm GN},q}$ such that 
\[
\|f\|_{L^q(B^c)}
\leq 
C_{{\rm GN},q}
\|\nabla f\|_{L^2(B^c)}^{1-\frac{1}{q}}
\|f\|_{L^1(B^c)}^{\frac{1}{q}}, 
\quad \forall f\in H_0^1(B^c)\cap L^1(B^c).
\]
\end{lemma}

To elicit the effect from the boundary (in two dimension), 
we also use the critical case of the Hardy inequality 
(cf. Ladyzhenskaya \cite{Ladyzhenskaya} and also Dan--Shibata \cite{DS1995}).
For the reader's convenience, we give a short proof.

\begin{lemma}\label{lem:critHardy}
If $f\in H_0^1(B^c)$, then
\[
\frac{1}{4}\int_{B^c}\frac{f^2}{|x|^2(1+\log|x|)^2}\,dx
\leq 
\int_{B^c}|\nabla f|^2\,dx.
\]
\end{lemma}

\begin{proof}
By density, it suffices to discuss the estimate for $f\in C_0^\infty(B^c)$.
To simplify the notation, we set $H(x)=1+\log |x|$ which is positive and harmonic in $B^c$.
Then by using the transform $g=H^{-\frac{1}{2}}f \in C_0^\infty(B^c)$ 
and integration by parts, 
we can calculate as follows:
\begin{align*}
\int_{B^c}|\nabla f|^2\,dx
&=
\int_{B^c}\Big(|\nabla g|^2H+g\nabla g\cdot \nabla H+\frac{|\nabla H|^2}{4H}g^2\Big)\,dx
\\
&=
\int_{B^c}|\nabla g|^2H\,dx 
+
\int_{B^c}\left(-\frac{\Delta H}{2H}+\frac{|\nabla H|^2}{4H^2}\right)f^2\,dx.
\end{align*} 
This gives the desired inequality.
\end{proof}

\section{Linear decay estimates for $g\in L_{\rm rad}^2\cap L_{d\mu}^1$}\label{sec:linear}
\paragraph{}

In this section we study 
the decay property of the solution $u(t)=S(t)g$ 
of the linear damped wave equation
\begin{equation}\label{eq:sec3:dw-linear}
\begin{cases}
\pa_t^2u(x,t)-\Delta u(x,t)+\pa_tu(x,t)=0
&\text{in}\ B^c\times (0,T), 
\\
u(x,t)=0
&\text{on}\ \pa B^c\times (0,T), 
\\
(u,\pa_tu)(x,0)=(0,g(x)), &\text{in}\ B^c.
\end{cases}
\end{equation}
under the additional integrability condition $g\in L_{d\mu}^1$. 
Actually, 
under the condition $g\in L_{d\mu}^1$, 
we can find how a compact obstacle affects 
to the behavior of solutions
to the damped wave equation. 
The harvests are different from the usual damped wave equation in the whole space. 

\subsection{A Matsumura estimate with logarithmic weight}

Here we state the decay estimate for solutions to 
the damped wave equation \eqref{eq:sec3:dw-linear} 
for the initial condition 
$(u,\pa_tu)(0)=(0,g)$ with $g\in L^2(B^c)\cap L_{d\mu}^1$.
We emphasize that 
as the effect from the boundary, 
the following Matsumura type estimates have 
the decay rate involving the logarithmic function. 
This fact can be shown via the use of 
the $L_{d\mu}^q$-$L^2$ estimates for the Dirichlet heat semigroup $e^{t\Delta_{B^c}}$. 
\begin{lemma}\label{lem:linear:est-grad}
For every $q\in [1,2]$, 
there exist positive constants 
$C_{{\rm M},1,q}^\sharp $
and $C_{{\rm M},2,q}^\sharp $
 such that 
if $g\in L^2(B^c)\cap L_{d\mu}^q$, then
for every $t>0$, 
\begin{align*}
\|\nabla S(t)g\|_{L^2(B^c)}
&\leq C_{{\rm M},1,q}^\sharp h(t)^{\frac{1}{q}}
(\|g\|_{L_{d\mu}^q}+\|g\|_{L^2(B^c)}), 
\\
\|\pa_t S(t) g\|_{L^2(B^c)}
&\leq 
C_{{\rm M},2,q}^\sharp (1+t)^{-\frac{1}{2}}h(t)^{\frac{1}{q}} 
(\|g\|_{L_{d\mu}^q}+\|g\|_{L^2(B^c)}).
\end{align*}
\end{lemma}
\begin{proof}
The case $0<t\leq 1$ is obvious via Lemma \ref{lem:energy-est}.
Let $t\geq 1$ be arbitrary. 
using Lemma \ref{lem:matsumura-applied}, we see by the notation $s=t/2$ that 
\begin{align*}
\|\nabla S(t)g\|_{L^2(B^c)}
&\leq 
\|\nabla e^{t\Delta_{B^c}}g\|_{L^2(B^c)}
+
\|\nabla (S(t)-e^{t\Delta_{B^c}})g\|_{L^2(B^c)}
\\
&\leq 
\frac{1}{s^{\frac{1}{2}}}\|e^{s\Delta_{B^c}}g\|_{L^2(B^c)}
+
\frac{C_{{\rm M},1}'}
{t^{\frac{3}{2}}}\|g\|_{L^2(B^c)}
\\
&\leq \frac{C_{{\rm H},q}}{s^{\frac{1}{q}}(1+\log (1+s))^{\frac{1}{q}}}\|g\|_{L_{d\mu}^q}
+
\frac{C_{{\rm M},1}'}{t^{\frac{3}{2}}}\|g\|_{L^2(B^c)}
\end{align*}
and also 
\begin{align*}
\|\pa_t S(t)g\|_{L^2(B^c)}
&\leq 
\|\pa_t e^{t\Delta_{B^c}}g\|_{L^2(B^c)}
+
\|\pa_t (S(t)-e^{t\Delta_{B^c}})g\|_{L^2(B^c)}
\\
&\leq 
\frac{1}{s}\|e^{s\Delta_{B^c}}g\|_{L^2(B^c)}
+
\frac{C_{{\rm M},2}'}{t^{2}}
\|g\|_{L^2(B^c)}
\\
&\leq \frac{C_{{\rm H},q}}{s^{\frac{1}{2}+\frac{1}{q}}(1+\log (1+s))^{\frac{1}{q}}}\|g\|_{L_{d\mu}^q}
+
\frac{C_{{\rm M},2}'}{t^{2}}\|g\|_{L^2(B^c)}.
\end{align*}
These imply the desired estimates.
\end{proof}

\subsection{A positively preserving 
property}

In this subsection, we study 
a positively preserving 
property of the solution map $S(t)$ for the linear damped wave equation \eqref{eq:sec3:dw-linear} under the radial symmetry.
\begin{proposition}\label{prop:mp}
If $g\in L_{\rm rad}^2$ is nonnegative, 
then the solution $S(t)g$ of \eqref{eq:sec3:dw-linear} 
is also radially symmetric and nonnegative for $t>0$.
\end{proposition}
\begin{proof}
It is enough to show the assertion for the functions 
belonging to $D(\Delta_{B^c})=H^2(B^c)\cap H_0^1(B^c)$.  
Indeed, since the resolvent operator $(1-\frac{1}{n}\Delta_{B^c})^{-1}$
preserves the nonnegativity and the radial symmetry 
(for nonnegativity see e.g, Brezis \cite[Section 9.7]{BrezisBook}, 
the radial symmetry is the consequence of rotation), 
if $g\in L_{\rm rad}^2$, then 
for each $n\in\N$, 
$g_n=(1-\frac{1}{n}\Delta_\Omega)^{-1}g\in D(\Delta_{B^c})$ 
is also nonnegative and radially symmetric. 
If the assertion for $D(\Delta_{B^c})$ holds, then 
we have the nonnegativity of $S(t)g_n$ for $n\in \N$. 
Noting that $g_n\to g$ in $L^2(B^c)$ as $n\to\infty$ 
and 
recalling $S(t)g_n=P_1e^{t\mathcal{L}_{B^c}}(0,g_n)$ with 
the projection $P_1(u,v)=u$, 
we see by the continuity of $e^{t\mathcal{L}_{B^c}}$ in $H_0^1(B^c)\times L^2(B^c)$ 
that $S(t)g_n\to S(t)g$  in $H_0^1(B^c)$ as $n\to\infty$  
which implies the nonnegativity and the radial symmetry of $S(t)g$. 

Now we assume $g \in D(\Delta_{B^c})=H^2(B^c)\cap H^1_0(B^c)$.
Then, the corresponding solution
$u=S(t)g$ has the regularity
\begin{gather*}
	u \in C([0,\infty) ; H^3(B^c)) \cap C^1([0,\infty) ; H^2(B^c)) \cap C^2([0,\infty) ; H^1(B^c))
\end{gather*}
with the boundary condition $u, \Delta u\in C([0,\infty) ; H_0^1(B^c))$
and satisfies the equation of \eqref{eq:dw-linear} in $H_0^1(B^c)$.
Since $g$ is radially symmetric, 
so is $u(\cdot,t)=S(t)g$ (as the consequence of rotation) 
and hence, noting 
$\Delta u=\pa_r^2u+\frac{1}{r}\pa_r u$
we can find that 
the new function 
$U(r,t) = e^{t/2} r^{1/2}u(x,t)$ with $r = |x|$ 
satisfies the equation
\begin{align*}
	\pa_t^2 U - \pa_r^2 U = \frac{1}{4} \left( \frac{1}{r^2} + 1 \right) U\quad\text{in}\ (1,\infty) \times (0,\infty)
\end{align*}
with  the initial condition $U(r,0)=0$ and 
$\pa_tU(r,0)=g_1(r)=r^{1/2}g(r)$.  
Then, the above regularity of $u$ shows 
$U \in C^2([1,\infty) \times [0,\infty))$
with $U(0,t)=\pa_r^2U(0,t)=0$,
which justifies the following argument 
in the classical sense.

By the reflection,
we extend $g_1(\cdot)$ and $U(\cdot,t)$ to be odd functions.
We set $\widetilde{g}_1$ and $\widetilde{U}$
as
\begin{align*}
	\widetilde{g}_1(y) = 
	\begin{cases}
	g_1(1+y) &(y > 0),\\ 0 &(y=0), \\ -g_1(1-y) &(y<0),
	\end{cases}
	\quad
	\widetilde{U}(y,t) = 
	\begin{cases}
	U(1+y,t) &(y > 0,t>0),\\ 0 &(y=0,t>0), \\ -U(1-y,t) &(y<0,t>0), 
	\end{cases}
\end{align*}
which satisfy the initial value problem 
\begin{align*}
\begin{cases}
\pa_t^2 \widetilde{U} - \pa_y^2 \widetilde{U} = m 
\widetilde{U}
& \text{in}\ \mathbb{R} \times (0,\infty),
\\
(\widetilde{U},\pa_t \widetilde{U})(y,0) = (0, \widetilde{g}_1(y))
& \text{in}\ \mathbb{R} \times \{t=0\},
\end{cases}
\end{align*}
where $m(y) = \frac{1}{4} \big( \frac{1}{(|y|+1)^2} + 1 \big)$. 
Note that $\widetilde{g}_1\in C^1(\R)$ and $\widetilde{U}\in C^2(\R\times (0,\infty))$
by virtue of the behavior of the boundary $g_1(1)=0$ and $U(1,t)=\pa_r^2U(1,t)=0$.

We prove the nonnegativity of $u$ by a contradiction argument similar to \cite[Chapter 4]{PrWe}.
We assume that there exists a point $(x_0,t_0)\in B^c\times (0,\infty)$ such that 
$u(x_0,t_0)<0$. Then clearly, setting $y_0=|x_0|-1>0$, we have 
\[
\widetilde U(y_0,t_0)=U(|x_0|,t_0)=e^{t_0/2}|x_0|^{\frac{1}{2}}u(x_0,t_0)<0. 
\]
Now we fix the parameter $\ep>0$ satisfying $\widetilde U(y_0,t_0)+\ep e^{t_0}=0$
and put
\begin{align*}
	V(y,t) = \widetilde{U}(y,t) + \varepsilon e^{t}, 
	\quad (y,t)\in \R\times (0,\infty).
\end{align*}
Then, the new function $V$ satisfies
\begin{align*}
\begin{cases}
\pa_t^2 V - \pa_y^2 V = m V + \ep ( 1-m) e^{t}
& \text{in}\ \mathbb{R} \times (0,\infty),
\\
(V,\pa_tV)(y,0) = (\ep , \widetilde{g}_1(y)+\ep)
& \text{in}\ \mathbb{R} \times \{t=0\}. 
\end{cases}
\end{align*}
Then we consider the triangular region 
\[
D_0= \{ (y,t) \in \mathbb{R} \times (0,\infty) ; t+|y-y_0| < t_0\}.
\]
By $V(y,0) = \varepsilon > 0$ and the continuity,
$V$ is positive in $\overline{D_0}$ with sufficiently small $t$.
Thus, by considering a zero of $V$ in $\overline{D_0 \cap \{ y > 0 \}}$ with the smallest time
(the zero set of $V$ in $\overline{D_0 \cap \{ y > 0 \}}$ is not empty due to 
$V(y_0, t_0)=0$),
there exists a point
$(y_1,t_1) \in \overline{D_0 \cap \{ y > 0 \}}$
such that
\begin{align*}
	V(y_1,t_1) = 0 \quad \mbox{and} \quad V > 0 \ \mbox{in} \ D_1 \cap \{ y > 0\},
\end{align*}
where
$D_1$ is the triangular region
$D_1 = \{ (y,t) \in \mathbb{R} \times (0,\infty) ; t+|y-y_1| < t_1\} (\subset D_0)$.
We further define the subregion
$D_1' = \{ (y,t) \in D_1 ;t+|y| < t_1- y_1\}$.
Note that
$D_1' = \emptyset$ if $t_1 \le y_1$
and
$D_1\setminus D_1'$
is the trapezoidal region with the vertices $(y_1,t_1), (0,t_1-y_1), (t_1-y_1,0), (y_1+t_1,0)$
if $t _1 > y_1$.
Applying the d'Alembert formula 
with the above notation, we have
\begin{align*}
	0=V(y_1,t_1)&=
	\frac{1}{2} \left( V(y_1-t_1,0) + V(y_1+t_1,0) \right)
	+ \frac{1}{2} \int_{y_1-t_1}^{y_1+t_1} \partial_t V(y,0) \,dy \\
	&\quad
	+ \frac{1}{2}\iint_{D_1} m(y) V(y,t) \,dydt 
	+ \frac{\varepsilon}{2}\iint_{D_1} ( 1- m(y) )  e^{t} \,dydt
\\
	&=
	(1+t_1)\ep + \frac{1}{2} \int_{|y_1-t_1|}^{y_1+t_1} \widetilde{g}_1(y) \,dy
	+ \frac{\ep}{2}\iint_{D_1'} m(y)e^{t} \,dydt
\\
&\quad
	+ \frac{1}{2}\iint_{D_1\setminus D_1'} m(y) V(y,t) \,dydt
	+ \frac{\varepsilon}{2}\iint_{D_1} ( 1- m(y) )  e^{t} \,dydt, 
\end{align*}
where we have used that
$\widetilde{g}$ and $\widetilde{U}(\cdot,t)$ are odd functions 
and $m$ is an even function.
Then using the conditions 
\begin{align*}
\widetilde{g}_1\geq 0\ \text{on}\ (0,\infty), 
\quad V>0\ \text{on}\ D_1 \cap \{ y > 0 \}, 
\quad 0\leq m\leq \frac{1}{2}\ \text{on}\ \R
\end{align*}
and the fact $D_1\setminus D_1' \subset D_1 \cap \{ y > 0 \}$,
we find that the right-hand side of the above identity is positive, 
which is contradiction.
The proof is complete.
\end{proof}

\subsection{An $L_{d\mu}^1$-estimate for $g\in L_{\rm  rad}^2\cap L_{d\mu}^1$}

By virtue of the positively preserving property of $S(t)$ 
on the radially symmetric functions, 
we can discuss the validity of $L_{d\mu}^1$-estimates for $S(t)g$. 
Basically, the following lemma 
is similar to the analysis of the ordinary differential 
equation $y''+y'=0$.  

\begin{lemma}\label{lem:linear:wL1est}
For every $g\in L_{\rm rad}^2\cap L_{d\mu}^1$, 
one has $S(t)g\in C([0,\infty);L_{d\mu}^1)
$  
with
\[
\|S(t)g\|_{L^1_{d\mu}}\leq (1-e^{-t})\|g\|_{L^1_{d\mu}}, \quad t>0.
\]
In particular, $S(t)$ can be extended to the bounded operator from $L^2_{\rm rad}\cap L_{d\mu}^1$ to itself. 
\end{lemma}
\begin{proof}
To shorten the notation, we use $u=S(t)g$. 
We divide the proof into three steps as follows:
\begin{description}
\item[Step 1.] the case where $0\leq g\in L_{\rm rad}^2$ having bounded support, 
\item[Step 2.] the case where $0\leq g\in L_{\rm rad}^2\cap L_{d\mu}^1$ without assumption on the support,
\item[Step 3.] the case where $g\in L_{\rm rad}^2\cap L_{d\mu}^1$ admitting the change of sign. 
\end{description}

\textbf{(Step 1)}
Let $R>1$ satisfy 
${\rm supp}\,g\subset B(0,R)=\{x\in\R^2\;;\;|x|< R\}$.
Then by finite propagation property, we have 
${\rm supp}\,u(t)\subset B(0,R+t)$. 
Here we fix $\zeta\in C^\infty(\R^2\times [0,\infty))$ as 
$\zeta(x,t)=\zeta_0((R+t)^2-|x|^2)$ 
with $\zeta_0\in C^\infty(\R)$ satisfying 
$\zeta_0\equiv 1$ on $[0,\infty)$ 
and $\zeta_0\equiv 0$ on $(-\infty,-1)$. 
Here we note that $\zeta(\cdot,t)\in C_0^\infty(\R^N)$ 
and $\zeta\equiv 1$ on ${\rm supp}\,u$. 
By the nonnegativity of $u(t)$  (provided by Proposition \ref{prop:mp}), 
we can see from 
$u\in C([0,\infty);L^2(B^c))$ 
and 
$\zeta(x,t)(1+\log |x|)\in C([0,\infty);L^2(B^c))$ that 
$u(x,t)=\zeta(x,t)u(x,t)\in C([0,\infty);L_{d\mu}^1)$. 
Moreover, Setting 
$\varphi_n(x)=1+\log |x|-|x|^{-n}$ (for $n\in \N$), 
we have $\varphi_n(x)\to 1+\log |x|$ as $n\to \infty$ 
and
$\zeta(x,t)\varphi_n(x)\in C^2([0,\infty);H_0^1(B^c))$
which is applicable to 
the test function for the equation in \eqref{eq:dw-linear} (verified in $H^{-1}(B^c)$). 
Hereafter, all integrals on $B^c$ always can be justified 
by using the relation $\zeta(x,t)u(x,t)=u(x,t)$.
Noting that 
\[
\Delta \varphi_n
=-\Delta |x|^{-n}
=n\,{\rm div}(x|x|^{-n-2})=-n^2|x|^{-n-2}\leq 0,
\]
we see from the nonnegativity of $u(t)$ that 
\begin{align*}
\lr{\Delta u(t),\zeta(t)\varphi_n}_{H^{-1},H_0^1}
=
\int_{B^c} 
u(t)\Delta\varphi_n\,dx\leq 0.
\end{align*}
Therefore the equation in \eqref{eq:dw-linear} 
gives 
\begin{align*}
\frac{d^2}{dt^2}
\int_{B^c} u(t)\varphi_n\,dx
+
\frac{d}{dt}
\int_{B^c} u(t)\varphi_n\,dx
=
\lr{\Delta u(t),\zeta(t)\varphi_n}_{H^{-1},H_0^1}
\leq 0.
\end{align*}
which implies 
\[
\int_{B^c} u(t)\varphi_n\,dx
\leq (1-e^{-t})\int_{B^c} g\varphi_n\,dx.
\]
Letting $n\to \infty$, we obtain the desired inequality. 

\textbf{(Step 2)}
We use 
the cut-off approximation 
$g_n=\chi_{B^c\cap B(0,n)}(x)g(x)$, 
where $\chi_K$ is the indicator function of $K$. 
Put $u_n=S(t)g_n$. Then 
$g_n\to g$ in $L^2(B^c)$ as $n\to \infty$
and hence $u_n(t) \to u(t)$ in $H_0^1(B^c)$ as $n\to \infty$.
For each $n\in \N$, we can apply the claim
in Step 1 to $g_n$ 
and also $g_m-g_n$ ($m>n$)
and then 
$u_n\in C([0,\infty);L_{d\mu}^1)$ and 
\begin{gather*}
\|u_n(t)\|_{L_{d\mu}^1}
\leq 
(1-e^{-t})
\|g_n\|_{L_{d\mu}^1}
\leq 
(1-e^{-t})
\|g\|_{L_{d\mu}^1},
\\
\|u_m(t)-u_n(t)\|_{L_{d\mu}^1}
\leq 
(1-e^{-t})
\|g_m-g_n\|_{L_{d\mu}^1}.
\end{gather*}
Therefore $u_n$ is the Cauchy sequence in the Banach space 
$C(I;L_{d\mu}^1)$ endowed with the sup norm 
for any compact interval $I\subset [0,\infty)$. 
Since $u_n(t)$ converses to $u(t)$ in the pointwise sense, 
we can obtain $u \in C([0,\infty);L_{d\mu}^1)$ and the desired inequality. 

\textbf{(Step 3)}
In this case we use 
the decomposition
\[
g=g_+-g_-, \quad g_\pm:=\max\{\pm g,0\}\geq 0.
\]
Then by the consequence of Step 2, we have
$S(t)g=S(t)g_+-S(t)g_-\in C([0,\infty);L_{d\mu}^1)$
and
\begin{align*}
\|S(t)g\|_{L_{d\mu}^1}
&\leq 
\|S(t)g_+\|_{L_{d\mu}^1}
+\|S(t)g_-
\|_{L_{d\mu}^1}
\\
&\leq 
(1-e^{-t})\Big(\|g_+\|_{L_{d\mu}^1}
+\|g_-
\|_{L_{d\mu}^1}\Big)
\\
&=
(1-e^{-t})\|g\|_{L_{d\mu}^1}.
\end{align*}
The proof is complete.
\end{proof}

\section{Estimates for semilinear problem}\label{sec:semilinear}

In this section, we discuss the estimate from below
for the lifespan of solution 
\[
u(t)=\ep S(t)g+\int_{0}^tS(t-s)[|u(s)|^p]\,ds, \quad t\in (0,T)
\]
to \eqref{intro:problem}. The argument depends on the continuity method 
based on the blowup alternative (Proposition \ref{prop:ndw:fundamental} {\bf (ii)}).
The quantity 
\begin{equation}\label{eq:X_Tnorm}
\|v\|_{X_T}=\sup_{0\leq t< T}
\left(
\|v(t)\|_{L^1_{d\mu}}
+
\frac{\|\nabla v(t)\|_{L^2(B^c)}}{h(t)}\right), 
\quad 
v\in C([0,T);H_0^1(B^c)\cap L_{d\mu}^1).
\end{equation}
plays a crucial role. We first note that the above quantity 
for the linear solution $S(t)g$ with $g\in L_{\rm rad}^2\cap L_{d\mu}^1$ 
is finite. 
\begin{lemma}\label{lem:X_T:linear}
There exists a positive constant $C_1$ such that 
for every $g\in L_{\rm rad}^2\cap L_{d\mu}^1$, 
\[
\|S(t)g\|_{X_\infty}
\leq 
C_1\Big(\|g\|_{L_{d\mu}^1}+\|g\|_{L^2(B^c)}\Big).
\]
\end{lemma}
\begin{proof}
We can choose $C_1=1+C_{{\rm M},1,1}^\sharp$
which is a consequence of Lemmas \ref{lem:linear:wL1est} and \ref{lem:linear:est-grad}.
\end{proof}

Before treating the lifespan of $u$, we state a modified assertion 
about the blowup alternative from the viewpoint of the quantity $\|u\|_{X_T}$.

\begin{lemma}\label{lem:alternative-X_T}
Let $u$ be the weak solution of \eqref{intro:problem} 
in $(0,T)$ with $g$ having a compact support. 
Then $T=T_{\max}(0,\ep g)$ 
if and only if 
$T=+\infty$ 
or 
$T<+\infty$ with $\lim_{t\to T}\|u\|_{X_t}=+\infty$. 
\end{lemma}
\begin{proof}
Suppose that $T_{\max}(0,\ep g)=T<+\infty$. 
If $\lim_{t\to T}\|u\|_{X_t}<+\infty$, then 
Lemma \ref{lem:GN} with $q=2$ yields
\[
\|u(t)\|_{L^2(B^c)}
\leq 
C_{{\rm GN},2}\|\nabla u(t)\|_{L^2(B^c)}^\frac{1}{2}
\|u(t)\|_{L^1(B^c)}^{\frac{1}{2}}
\leq C_{{\rm GN},2}\|u\|_{X_T}h(t)^{\frac{1}{2}}
\]
and 
Lemmas \ref{lem:energy-est} and \ref{lem:GN} with $q=2p$
give for every $t\in (0,T)$,
\begin{align*}
\|\pa_tu(t)\|_{L^2(B^c)}
&=
\left\|\int_0^t\pa_tS(t-s)|u(s)|^p\,ds\right\|_{L^2(B^c)}
\\
&\leq 
\int_0^t\|u(s)\|_{L^{2p}(B^c)}^p\,ds
\\
&\leq 
(C_{{\rm GN},2p})^p\int_0^t\|\nabla u(s)\|_{L^{2}(B^c)}^{p-\frac{1}{2}}\|u(s)\|_{L^1(B^c)}^{\frac{1}{2}}\,ds
\\
&\leq 
(C_{{\rm GN},2p})^p\|u\|_{X_T}^p\int_0^t h(s)^{p-\frac{1}{2}}\,ds.
\end{align*}
These inequalities yield 
$\|u\|_{H_0^1(B^c)}+\|\pa_t u\|_{L^2(B^c)}$ is bounded in $(0,T)$.
By Proposition \ref{prop:ndw:fundamental}, we have $T<T_{\max}(0,\ep g)$ which is contradiction. 

On the contrary, 
suppose $T<+\infty$ with $\lim_{t\to T}\|u\|_{X_t}=+\infty$. 
We fix $R>1$ such that ${\rm supp}\,g\in B(0,R)$. By finite propagation property we also 
have ${\rm supp}\,u(t)\subset B(0,R+T)$. We see from the H\"older inequality, 
that 
\begin{align*}
\|u(t)\|_{L_{d\mu}^1}
&=\int_{B^c\cap B(0,R+T)}|u(t)|(1+\log |x|)\,dx
\\
&\leq \left(
\int_{B^c\cap B(0,R+T)}(1+\log |x|)^2\,dx\right)^\frac{1}{2}
\|u(t)\|_{L^2(B^c)}
\end{align*}
and hence $\lim_{t\to T}\|u(t)\|_{H_0^1(B^c)}=+\infty$
which means that $T=T_{\max}(0,\ep g)$. 
\end{proof}

\subsection{A Gagliardo--Nirenberg type inequality with logarithmic weight}

As explained in Introduction, to 
estimate the $L_{d\mu}^1$-norm of the solution $u$ to the problem \eqref{intro:problem}, 
we have to control the quantity 
\[
\big\||u|^p\big\|_{L_{d\mu}^1}=\int_{B^c}|u(x,t)|^p(1+\log |x|)\,dx
\]
which comes from the semilinear term $|u|^p$.
The following inequality enables us to control such a quantity 
via the ingredients in $\|u\|_{X_T}$. 
The form is close to the classical Gagliardo--Nirenberg inequality 
but involving the logarithmic weight. 
Here we do not need to assume the radial symmetry for functions.

\begin{lemma}\label{lem:logGN}
For every $1<q<\infty$, 
one has $H^1_0(B^c)\cap L^1_{d\mu} \subset L^q_{d\mu}$ and 
there exists 
a positive constant $C_{{\rm GH},q}^\sharp$ such that 
\[
\|f\|_{L_{d\mu}^q}
\leq 
C_{{\rm GN},q}^\sharp\|\nabla f\|_{L^2(B^c)}^{1-\frac{1}{q}}
\|f\|_{L_{d\mu}^1}^{\frac{1}{q}}, 
\quad 
f\in H^1_0(B^c)\cap L^1_{d\mu}.
\]
\end{lemma}

\begin{proof}
Let $f\in H^1_0(B^c)\cap L^1_{d\mu}$ be fixed.
Put $\eta\in C_0^\infty(\R)$ satisfying
\[
\eta(s)
\begin{cases}
>0&\text{if}\ s\in (1/2,2), 
\\
=0&\text{if}\ s\notin (1/2,2).
\end{cases}
\]
Note that for every $a>0$ and $\sigma\in [1,\infty)$, 
\begin{equation}
\label{eq:unitdecomp}
\int_{0}^\infty\eta\left(\frac{a}{R}\right)^\sigma\,dR
=a K_\sigma, 
\quad 
K_\sigma=\int_{0}^\infty\frac{\eta(\rho)^\sigma}{\rho^2}\,d\rho<\infty.
\end{equation}
Now we define the (localized) functions 
$f_R\in H^1_0(B^c)\cap L^1(B^c)$
for $R>0$ as 
\[
f_R(x)=\eta\left(\frac{1+\log |x|}{R}\right)f(x), \quad x\in B^c.
\]
We see from 
the Fubini--Tonelli theorem and 
\eqref{eq:unitdecomp} that for $1\leq \sigma <\infty$, 
\begin{align}
\label{eq:intf_R}
\int_0^\infty\|f_R\|_{L^\sigma (B^c)}^\sigma \,dR
&=
\int_{B^c} 
|f(x)|^\sigma 
\int_0^\infty\eta\left(\frac{1+\log |x|}{R}\right)^\sigma\,dR\,dx
=
K_\sigma 
\|f\|_{L_{d\mu}^\sigma}^\sigma.
\end{align}
On the other hand, by Lemma \ref{lem:GN} 
we have
\begin{align}
\label{eq:Nash}
\|f_R\|_{L^q(B^c)}^q
\leq (C_{{\rm GN},q})^q
\|\nabla f_R\|_{L^2(B^c)}^{q-1}
\|f_R\|_{L^1(B^c)}.
\end{align}
Noting that 
$\frac{1}{2}R\leq 1+\log|x| \leq 2R$
on ${\supp}\,f_R$, 
one can compute as
\begin{align*}
|\nabla f_R(x)|^2
&=
\left|
\eta\left(\frac{1+\log |x|}{R}\right)
\nabla f(x)
+
\eta'\left(\frac{1+\log |x|}{R}\right)\frac{x}{R|x|^2}f(x)
\right|^2
\\
&\leq 
2\|\eta\|_{L^\infty}^2|\nabla f(x)|^2
+
8\|\eta'\|_{L^\infty}^2
\frac{|f(x)|^2}{|x|^2(1+\log |x|)^2}.
\end{align*}
Combining the above inequality with Lemma \ref{lem:critHardy}, we deduce
\begin{align}
\label{eq:crit-Hardy}
\|\nabla f_R\|_{L^2(B^c)}^2\leq 
\left(2\|\eta\|_{L^\infty}^2+
32\|\eta'\|_{L^\infty}^2
\right)
\|\nabla f\|_{L^2(B^c)}^2.
\end{align}
Consequently, 
using \eqref{eq:Nash}, \eqref{eq:crit-Hardy} 
and \eqref{eq:intf_R} with $\sigma=q$ and also $\sigma=1$, we obtain
\begin{align*}
K_q\|f\|_{L_{d\mu}^q}^q
&=\int_0^\infty \|f_R\|_{L^q(B^c)}^q\,dR
\\
&\leq (C_{{\rm GN},q})^q
\int_0^\infty \|\nabla f_R\|_{L^2(B^c)}^{q-1}\|f_R\|_{L^1(B^c)}
\,dR
\\
&\leq 
(C_{{\rm GN},q})^q
\left(2\|\eta\|_{L^\infty}^2+
32\|\eta'\|_{L^\infty}^2
\right)^{\frac{q-1}{2}}\|\nabla f\|_{L^2(B^c)}^{q-1}
\int_0^\infty \|f_R\|_{L^1(B^c)}
\,dR
\\
&= 
K_1
(C_{{\rm GN},q})^q
\left(2\|\eta\|_{L^\infty}^2+
32\|\eta'\|_{L^\infty}^2
\right)^{\frac{q-1}{2}}
\|\nabla f\|_{L^2(B^c)}^{q-1}\|f\|_{L_{d\mu}^1}.
\end{align*}
The proof is complete.
\end{proof}

\subsection{Lower bound for lifespan of semilinear problem}

Here we discuss a priori estimate for $u$ via the quantity $\|u\|_{X_T}$. 
The first part is the derivation of $L_{d\mu}^1$-estimate 
as a merit of the positively preserving property of $S(t)$ 
for the radially symmetric functions. 
\begin{lemma}\label{lem:semilinear:wL1est}
Let $u$ be the weak solution 
of \eqref{intro:problem} 
in $(0,T)$ 
with 
$g\in L_{\rm rad}^2\cap L_{d\mu}^1$ 
and let $\|u\|_{X_T}$ be given in \eqref{eq:X_Tnorm}. 
Then there exists a positive constant $C_2$ 
(independent of $\ep, g$ and $T$) such that 
for every $t\in (0,T)$,
\[
\|u(t)\|_{L^1_{d\mu}}\leq 
\ep\|g\|_{L^1_{d\mu}}
+
C_2\|u\|_{X_T}^p\int_0^t{h(s)}^{p-1}\,ds. 
\]
\end{lemma}
\begin{proof}
Employing Lemma \ref{lem:logGN} with $q=p$, we have
for every $s\in (0,T)$, 
\begin{align*}
\|u(s)\|_{L^p_{d\mu}}^p
&\leq (C_{{\rm GN},p}^\sharp)^p\|\nabla u(s)\|_{L^2}^{p-1}
\|u(s)\|_{L^1_{d\mu}}
\\
&\leq (C_{{\rm GN},p}^\sharp)^p\|u\|_{X_T}^p h(s)^{p-1}.
\end{align*}
Combining the above inequality 
with Lemma \ref{lem:linear:wL1est}, 
we deduce that for every $t\in (0,T)$, 
\begin{align*}
\|u(t)\|_{L^1_{d\mu}}
&\leq \ep \|S(t)g\|_{L_{d\mu}^1}
+\int_0^t\|S(t-s)|u(s)|^p\|_{L_{d\mu}^1}\,ds
\\
&\leq \ep \|g\|_{L_{d\mu}^1}
+\int_0^t\|u(s)\|_{L_{d\mu}^p}^p\,ds
\\
&\leq \ep \|g\|_{L_{d\mu}^1}
+(C_{{\rm GN},p}^\sharp)^p\|u\|_{X_T}^p\int_0^t h(s)^{p-1}
\,ds.
\end{align*}
This is the desired inequality. 
\end{proof}
\begin{lemma}\label{lem:semilinear:est-grad}
Let $u$ be the weak solution of \ref{intro:problem}
in $(0,T)$ 
with 
$g\in L_{\rm rad}^2\cap L_{d\mu}^1$ 
and let $\|u\|_{X_T}$ be given in \eqref{eq:X_Tnorm}. 
Then there exists a positive constant $C_3$ 
(independent of $\ep, g$ and $T$) such that 
for every $t\in (0,T)$,
\[
\|\nabla u(t)\|_{L^2(B^c)}\leq 
C_3h(t)
\left(
\ep (\|g\|_{L^1_{d\mu}}+\|g\|_{L^2(B^c)})
+
\|u\|_{X_T}^p\int_0^th(s)^{p-1}\,ds 
\right).
\]
\end{lemma}
\begin{proof}
Set 
\[
J_1(t)=
\int_0^\frac{t}{2}S(t-s)|u(s)|^p\,ds, \quad 
J_2(t)=
\int_\frac{t}{2}^t S(t-s)|u(s)|^p\,ds.
\]
Then it is obvious that 
$u(t)=\ep  S(t)g+J_1(t)+J_2(t)$. 
The estimate for the linear part $\ep S(t)g$ has already done in Lemma \ref{lem:X_T:linear}. 
For the estimate of $J_1(t)$, 
applying Lemma \ref{lem:linear:est-grad} 
with $q=1$, we see that 
\begin{align*}
\|\nabla J_1(t)\|_{L^2(B^c)}
&\leq 
\int_0^{\frac{t}{2}}
\|\nabla S(t-s)|u(s)|^p\|_{L^2(B^c)}\,ds
\\
&\leq
C_{{\rm M},1,1}^\sharp\int_0^{\frac{t}{2}}h(t-s)
\Big(\|u(s)\|_{L^p_{d\mu}}^p+\|u(s)\|_{L^{2p}(B^c)}^{p}\Big)\,ds.
\end{align*}
The integrands in the right-hand side of the above inequality 
are estimated as follows: 
via Lemma \ref{lem:logGN} with $q=p$, one has
\begin{align*}
\|u(s)\|_{L^p_{d\mu}}^p
&\leq 
(C_{{\rm GN},p}^\sharp)^p
\|\nabla u(s)\|_{L^2(B^c)}^{p-1}\|u(s)\|_{L^1_{d\mu}}
\\
&\leq 
(C_{{\rm GN},p}^\sharp)^p\|u\|_{X_T}^ph(s)^{p-1}
\end{align*}
and via Lemma \ref{lem:GN} with $q=2p$, one has
\begin{align*}
\|u(s)\|_{L^{2p}(B^c)}^{p}
&\leq 
(C_{{\rm GN},2p})^{p}
\|\nabla u(s)\|_{L^2(B^c)}^{p-\frac{1}{2}}\|u(s)\|_{L^1(B^c)}^{\frac{1}{2}}
\\
&\leq 
(C_{{\rm GN},2p})^{p}
\|u\|_{X_T}^p
h(s)^{p-\frac{1}{2}}.
\end{align*}
Noting that $h(s)^{p-\frac{1}{2}}\leq h(s)^{p-1}$, we can deduce that 
\begin{align*}
\|\nabla J_1(t)\|_{L^2(B^c)}
&\leq 
C_{{\rm M},1,1}^\sharp
\big((C_{{\rm GN},p}^\sharp)^p+(C_{{\rm GN},2p})^p\big)
\|u\|_{X_T}^p\int_0^{\frac{t}{2}}h(t-s)h(s)^{p-1}\,ds
\\
&\leq 
C_{{\rm M},1,1}^\sharp
\big((C_{{\rm GN},p}^\sharp)^p+(C_{{\rm GN},2p})^p\big)
\|u\|_{X_T}^ph(t/2)\int_0^{\frac{t}{2}}h(s)^{p-1}\,ds.
\end{align*}
For the estimate of $J_2(t)$, 
applying Lemma \ref{lem:linear:est-grad} 
with $q=2$
and Lemma \ref{lem:logGN} with $q=2p$, 
we can also compute in a similar way:
\begin{align*}
\|\nabla J_2(t)\|_{L^2(B^c)}
&\leq 2C_{{\rm M},1,2}^\sharp
\int_{\frac{t}{2}}^t
h(t-s)^{\frac{1}{2}}\|u(s)\|_{L_{d\mu}^{2p}}^p\,ds
\\
&\leq 
2C_{{\rm M},1,2}^\sharp(C_{{\rm GN},2p}^\sharp)^p\|u\|_{X_T}^p
\int_{\frac{t}{2}}^th(t-s)^{\frac{1}{2}}h(s)^{p-\frac{1}{2}}\,ds
\\
&\leq 
k_0 C_{{\rm M},1,2}^\sharp(C_{{\rm GN},2p}^\sharp)^p\|u\|_{X_T}^p
t
h(t/2)^{p}
\\
&\leq 
2k_0 C_{{\rm M},1,2}^\sharp(C_{{\rm GN},2p}^\sharp)^p\|u\|_{X_T}^p
h(t/2)\int_0^{\frac{t}{2}}h(s)^{p-1}\,ds,
\end{align*}
where we have used the inequality
\[
\int_0^\tau h(s)^{\frac{1}{2}}\,ds\leq k_0\tau h(\tau)^{\frac{1}{2}}, \quad \tau\in (0,\infty)
\]
for some $k_0>0$.
Combining these inequalities and 
noting $h(t/2)\leq 4h(t)$, 
we arrive at
the desired estimate.
\end{proof}

Summarizing the above two lemmas, we conclude the following
\begin{proposition}\label{prop:best-est}
Let $u$ be the weak solution of \ref{intro:problem}
in $(0,T)$ 
with 
$g\in L_{\rm rad}^2\cap L_{d\mu}^1$ 
and let $\|u\|_{X_T}$ be given in \eqref{eq:X_Tnorm}. 
Then there exists a positive constant $C_4$ 
(independent of $\ep, g$ and $T$) such that 
for every $t\in (0,T)$,
\begin{align*}
\|u\|_{X_t}
\leq C_4\left(\ep \big(\|g\|_{L_{d\mu}^1}+\|g\|_{L^2(B^c)}\big)+
\|u\|_{X_t}^p
\int_0^Th(s)^{p-1}\,ds\right).
\end{align*}
\end{proposition}

\begin{proof}[Proof of Theorem \ref{thm:main1}]

Put $|\!|\!|g|\!|\!|=\|g\|_{L_{d\mu}^1}+\|g\|_{L^2(B^c)}$ to 
shorten the notation. 
We first prove the assertion in Theorem \ref{thm:main1} 
when $g$ has a compact support. 
In this case, a similar discussion with $\zeta(x,t)$ as in Lemma \ref{lem:linear:wL1est}, 
we can find that $u\in C([0,T);L_{d\mu}^1)$ 
and therefore the function $t\in (0,T_{\max}(0,\ep g))\mapsto \|u\|_{X_t}$ is non-decreasing and continuous. 
In view of Lemma \ref{lem:alternative-X_T}, 
we can take
\[
T_*=\sup\{t\in (0,T_{\max}(0,\ep g)]\;;\;\|u\|_{X_{t}}\leq 2C_4 \ep |\!|\!|g|\!|\!|\}.
\]
Then by Proposition \ref{prop:best-est} we have
\[
\|u\|_{X_{T_*}}
\leq C_4\ep |\!|\!|g|\!|\!| 
\left(1+
(2C_4)^p|\!|\!|g|\!|\!|^{p-1}\ep^{p-1}
\int_0^{T_*}h(s)^{p-1}\,ds\right).
\]
This yields that $T_*$ has the lower bound $T_*\geq T_\ep$ with 
\begin{equation}\label{eq:Tep}
T_\ep =\sup 
\left\{t\in (0,\infty]\;;\;
(2C_4)^p|\!|\!|g|\!|\!|^{p-1}
\int_0^{t}h(s)^{p-1}\,ds\leq \frac{1}{\ep^{p-1}}
\right\}.
\end{equation}
Since
\[
\int_0^{t}h(s)^{p-1}\,ds
\begin{cases}
\leq k_p(1+t)^{2-p}(1+\log (1+t))^{1-p} &\text{if}\ 1<p<2, 
\\
=\log(1+\log (1+t)) &\text{if}\ p=2,
\\
<+\infty&\text{if}\ p>2
\end{cases}
\]
(for some positive constants $k_p$), 
the lifespan estimates for $1<p\leq 2$ and 
the existence of global weak solution of \eqref{intro:problem} are proved. 

Next we consider the case where 
$g\in L^2_{\rm rad}\cap L_{d\mu}^1$ does not have compact support.  
In this case, we use a cut-off argument. 
Put $g_n=g\chi_{B^c\cap B(0,n)}$. Note that 
$|\!|\!|g_{n}|\!|\!|\leq |\!|\!|g|\!|\!|$.
By the first step, we have 
the respective solutions 
\[
u_n\in C([0,T_\ep);H^1_0(B^c))\cap C^1([0,T_\ep);L^2(B^c))\cap C([0,T_\ep);L_{d\mu}^1)
\]
satisfying 
\[
\|u_n\|_{X_{T_{\ep}}}\leq 2C_4\ep |\!|\!|g_{n}|\!|\!|\leq 2C_4\ep |\!|\!|g|\!|\!|;
\]
it should be noticed that the $T_\ep$ can be chosen as in \eqref{eq:Tep} 
(independent of $n$).
Noting that for every $1\leq q<\infty$, 
we see that 
there exist  positive constants $C_{5,q}$ and $C_{5,q}'$ such that 
\[
\|u_n(t)\|_{L_{d\mu}^q}\leq C_{5,q}\ep , 
\quad
\|u_n(t)-u_m(t)\|_{L_{d\mu}^q}\leq C_{5,q}' y(t),
\]
where 
\[
y(t)=\|u_n(t)-u_m(t)\|_{L_{d\mu}^1}+
\|\nabla u_n(t)-\nabla u_m(t)\|_{L^2(B^c)}.
\]
Therefore 
we see from 
Lemmas \ref{lem:linear:wL1est} and \ref{lem:energy-est}
with 
the inequality $||z_1|^p-|z_2|^p|\leq p(|z_1|+|z_2|)^{p-1}|z_1-z_2|\ (z_1,z_2\in \R)$ that 
\begin{align*}
&\|u_n(t)-u_m(t)\|_{L_{d\mu}^1}
\\
&\leq 
\ep \|g_n-g_m\|_{L_{d\mu}^1}
+
\int_{0}^t
\big\||u_n(s)|^p-|u_m(s)|^p]\big\|_{L_{d\mu}^1}\,ds
\\
&\leq 
\ep \|g_n-g_m\|_{L_{d\mu}^1}
+
p\int_{0}^t
\big(\|u_n(s)\|_{L_{d\mu}^p}+\|u_m(s)\|_{L_{d\mu}^p}\big)^{p-1}
\|u_n(s)-u_m(s)\|_{L_{d\mu}^p}\,ds
\\
&\leq 
\ep \|g_n-g_m\|_{L_{d\mu}^1}
+
p
(2C_{5,p}\ep )^{p-1}C_{5,p}'\int_{0}^ty(s)\,ds
\end{align*}
and 
\begin{align*}
&\|\nabla u_n(t)-\nabla u_m(t)\|_{L^2(B^c)}
\\
&\leq 
\ep \|g_n-g_m\|_{L^2(B^c)}
+
\int_{0}^t
\big\||u_n(s)|^p-|u_m(s)|^p]\big\|_{L^2(B^c)}\,ds
\\
&\leq 
\ep \|g_n-g_m\|_{L^2(B^c)}
+
p
\int_{0}^t
\big(\|u_n(s)\|_{L_{d\mu}^{2p}}+\|u_m(s)\|_{L_{d\mu}^{2p}}\big)^{p-1}
\|u_n(s)-u_m(s)\|_{L_{d\mu}^{2p}}\,ds
\\
&\leq 
\ep \|g_n-g_m\|_{L_{d\mu}^1}
+
p
(2C_{5,2p}\ep)^{p-1}C_{5,2p}'\int_{0}^ty(s)\,ds.
\end{align*}
These imply 
\[
y(t)\leq \ep |\!|\!|g_n-g_m|\!|\!|
+
C_{6,p}\ep^{p-1}
\int_{0}^ty(s)\,ds,\quad t\in [0,T_\ep)
\]
with $C_{6,p}=p(2C_{5,p})^{p-1}C_{5,p}'+p(2C_{5,2p})^{p-1}C_{5,2p}'$. 
By the Gronwall inequality we obtain
\[
y(t)\leq \ep |\!|\!|g_n-g_m|\!|\!|e^{C_{6,p}\ep^{p-1}t}, \quad t\in [0,T_\ep)
\]
which ensures that $\{u_n\}_n$ is a Cauchy sequence in 
the Banach space $C(I;H_0^1(B^c)\cap L_{d\mu}^1)$ for any compact interval $I\subset [0,T_\ep)$. 
Since the limit $u$ satisfies 
\[
u(t)=S(t)g+\int_0^tS(t-s)|u(s)|^p\,ds,  
\quad t\in [0,T_\ep),  
\] 
the proof of this part ($g$ with non-compact support) is complete.

It only remains to show the decay estimate for $\pa_tu$ when 
$u$ is the global weak solution of \eqref{intro:problem} 
with $p>2$ and $|\!|\!|g|\!|\!|$ is sufficiently small verifying $T_\ep=\infty$. 
Then 
using Lemma \ref{lem:linear:est-grad}, we have
\begin{align*}
\|\pa_tu(t)\|_{L^2(B^c)}
&
\leq 
\ep \|\pa_tS(t)g\|_{L^2(B^c)}
+
\int_{0}^t\|\pa_t S(t-s)|u(s)|^p\|_{L^2(B^c)}\,ds
\\
&
\leq 
C_{{\rm M},2,1}^\sharp \ep (1+t)^{-\frac{1}{2}}h(t)|\!|\!|g|\!|\!|
+
C_{{\rm M},2,1}^\sharp
\int_{0}^{\frac{t}{2}}(1+t-s)^{-\frac{1}{2}}h(t-s)|\!|\!||u(s)|^p|\!|\!|\,ds
\\
&\quad +
2C_{{\rm M},2,2}^\sharp
\int_{\frac{t}{2}}^t(1+t-s)^{-\frac{1}{2}}h(t-s)^{\frac{1}{2}}\||u(s)|^p\|_{L_{d\mu}^2}\,ds.
\end{align*}
The rest of the proof of the boundedness is similar to the proof of Lemma \ref{lem:semilinear:est-grad} (the difference is just 
the validity of $\|u\|_{X_\infty}\leq 2C_4\ep |\!|\!|g|\!|\!|$). 
The proof is complete.
\end{proof}

\section*{Appendix}
\renewcommand{\thesection}{A}
\setcounter{theorem}{0}

Here we give an alternative proof of 
the $L^p$-$L^q$ type estimate (Lemma \ref{lem:L^p-wL^1}) 
involving the logarithmic weight 
for the Dirichlet heat semigroup $e^{t\Delta_{B^c}}$,
which describes the peculiarity of the two-dimensional exterior domain. 
Here we shall discuss it via the classical comparison principle for parabolic equations. 
Although all statements here can be shown for general exterior domains, 
we only pay our attention to the typical case $B^c$.  
A similar treatment also can be found in Sobajima \cite{Sobajima2021}.
\begin{lemma}\label{lem:apdx:1}
For every $q\in [1,\infty]$, there exists 
a positive constant $C_{{\rm A},1,q}$ such that 
if $f\in L^q(B^c)$, then
\[
|e^{t\Delta_{B^c}}f(x)|\leq
\frac{C_{{\rm A},1,q}\|f\|_{L^q(B^c)}}{t^{\frac{1}{q}}(1+\log(1+t))}
(1+\log |x|), \quad (x,t)\in B^c\times (0,\infty).
\]
\end{lemma}
\begin{proof}
Since $e^{t\Delta_{B^c}}$ is a  
positive operator, 
we may assume $f\geq 0$ without loss of generality.  
The standard (two-dimensional) $L^\infty$-$L^q$ estimate shows
for every $(x,t)\in B^c\times (0,\infty)$, 
\[
e^{t\Delta_{B^c}}f(x)
\leq e^{t\Delta_{\R^2}}f(x)
\leq 
\frac{\kappa_q\|f\|_{L^q(B^c)}}{t^{\frac{1}{q}}}, 
\quad 
\kappa_q=
\Big(1-\frac{1}{q}\Big)^{1-\frac{1}{q}}
\frac{1}{(4\pi)^{\frac{1}{q}}}
\]
which immediately gives for every $0<t\leq \tau=4$, 
\[
e^{t\Delta_{B^c}}f(x)\leq 
\frac{(1+\log 5)\kappa_q
\|f\|_{L^q(B^c)}}{t^{\frac{1}{q}}(1+\log (1+t))}(1+\log |x|), 
\quad x\in B^c.
\]
Therefore we focus our attention to the case $t\geq \tau=4$. 
Put $\mathcal{C}_t=B(0,t^{1/2})(\supset B)$ and 
\[
\mathcal{Q}_1=\bigcup_{t\geq \tau}(B_c\cap \mathcal{C}_t)\times \{t\}, 
\quad 
\mathcal{Q}_2=(B^c\times[\tau,\infty))\setminus \mathcal{Q}_1.
\]
We see from $1+\log(1+t)\leq 2(1+\log|x|)$ on 
$\mathcal{Q}_2$ that 
\begin{align*}
e^{t\Delta_{B^c}}f(x)
&\leq 
\frac{\kappa_q\|f\|_{L^q(B^c)}}{t^{\frac{1}{q}}}
\times \frac{2(1+\log |x|)}{
1+\log (1+t)}, 
\quad 
(x,t)\in \mathcal{Q}_2.
\end{align*}
For the estimate on $\mathcal{Q}_1$, we employ the comparison principle. 
Put for $(x,t)\in \overline{\mathcal{Q}_1}$,
\[
\Phi(x,t)=\frac{1+\log |x|^2}{2+\log t+\log |x|^2}, 
\quad 
U(x,t)=
\frac{\Phi(x,t)}{t^{\frac{1}{q}}}e^{-\frac{|x|^2}{4t}}.
\]
Observing that 
\begin{align*}
\pa_t\Phi=
-\frac{1+\log |x|^2}{t\Theta(x,t)^2},
\quad
\nabla\Phi=
\frac{2(1+\log t)x}{\Theta(x,t)^2|x|^2},
\quad
\Delta\Phi=-
\frac{8(1+\log t)}{\Theta(x,t)^3|x|^2}
\end{align*}
with $\Theta(x,t)=2+\log t+\log |x|^2$, 
we can deduce 
\begin{align*}
\pa_tU-\Delta U
\geq 
t^{-1-\frac{1}{q}}e^{-\frac{|x|^2}{4t}}
\left[\frac{1+2(\log t-\log |x|)}{\Theta^2}
+\Big(1-\frac{1}{q}\Big)\Phi\right]\geq 0, 
\quad (x,t)\in \mathcal{Q}_1.
\end{align*}
Moreover, noting that 
$\frac{1}{2+\log t}\leq \Phi(x,t)\leq \frac{1}{2}$ on $\mathcal{Q}_1$, 
we can find 
the comparison 
on the parabolic boundary of $\mathcal{Q}_1$ as follows:
\begin{align*}
\begin{cases}
e^{\tau\Delta}f(x)
\leq 
4e^{\frac{1}{4}}(1+\log 2)\kappa_q\|f\|_{L^q(B^c)}U(x,\tau), 
&x\in B^c \cap \mathcal{C}_{\tau}, 
\\
e^{t\Delta}f(x)=0\leq U(x,t), 
&x\in \pa \mathcal{C}_t, \ t\geq \tau, 
\\
e^{t\Delta}f(x)
\leq 2e^{\frac{1}{4}}\kappa_q\|f\|_{L^q(B^c)}U(x,t).
&x\in \pa B^c,\ t\geq \tau.
\end{cases}
\end{align*}
Therefore the comparison principle shows that 
\[
e^{t\Delta}f(x)
\leq 
4e^{\frac{1}{4}}(1+\log 2)\kappa_q
\|f\|_{L^q(B^c)}U(x,t), \quad 
(x,t)\in \mathcal{Q}_1. 
\]
The proof is complete.
\end{proof}

Taking the adjoint in Lemma \ref{lem:apdx:1}, 
we have
\begin{lemma}\label{lem:apdx:2}
For every $q\in [1,\infty]$, there exists 
a positive constant $C_{{\rm A},2,q}$ such that 
if $f\in L_{d\mu}^1$, then 
\[
\|e^{t\Delta_{B^c}}f\|_{L^q(B^c)}
\leq \frac{C_{{\rm A},2,q}}{t^{1-\frac{1}{q}}(1+\log (1+t))}\|f\|_{L_{d\mu}^1},
\quad t>0.
\]
\end{lemma}
\begin{proof}
Taking $g\in L^{q'}(B^c)$ with $\frac{1}{q}+\frac{1}{q'}=1$, 
we have
\begin{align*}
\int_{B^c} (e^{t\Delta_{B^c}}f)g\,dx
=
\int_{B^c} f(e^{t\Delta_{B^c}}g)\,dx
\leq 
\frac{C_{{\rm A},1,q'}\|g\|_{L^{q'}(B^c)}}{t^{\frac{1}{q'}}(1+\log (1+t))}
\int_{B^c} |f|(1+\log |x|)\,dx.
\end{align*}
Since $g$ is arbitrary, we obtain the desired inequality.
\end{proof}

The following is $L^p$-$L^q$ type estimate
with the logarithmic weight 
which is essentially used
in this paper. 
\begin{lemma}\label{lem:apdx:3}
For every $p\in [1,\infty)$ and $q\in [p,\infty]$,
there exists a positive constant 
$C_{{\rm A},3,p,q}$ such that 
if $f\in L_{d\mu}^p$, then 
\[
\|e^{t\Delta_{B^c}} f\|_{L^q(B^c)}\leq 
\frac{C_{{\rm A},3,p,q}}{t^{\frac{1}{p}-\frac{1}{q}}(1+\log(1+t))^{\frac{1}{p}}}\|f\|_{L_{d\mu}^p}, \quad t>0.
\]
\end{lemma}
\begin{proof}
The case $p=1$ is already proved in Lemma \ref{lem:apdx:2}. 
Let $1<p<\infty$ and $t>0$ be fixed. 
Then 
combining the $L^\infty$-contraction property 
for $e^{t\Delta_{B^c}}$ written as
\[
\|e^{t\Delta_{B^c}}f\|_{L^\infty(B^c)}
\leq 
\|f\|_{L^\infty(B^c)}, \quad f\in L^\infty(B^c)
\]
with the inequality  Lemma \ref{lem:apdx:2} of the form
\[
\|e^{t\Delta_{B^c}}f\|_{L^{\frac{q}{p}}(B^c)}
\leq 
\frac{C_{{\rm A},2,\frac{q}{p}}}{t^{1-\frac{p}{q}}(1+\log (1+t))}\|f\|_{L_{d\mu}^1}, 
\quad 
f\in L_{d\mu}^1,
\]
we see from the Riesz--Thorin theorem with the parameter 
$\theta\in (0,1)$ satisfying 
\[
\frac{\theta}{\infty}
+
\frac{(1-\theta)p}{q}=\frac{1}{r_1}, 
\quad 
\frac{\theta}{\infty}
+
\frac{1-\theta}{1}=\frac{1}{r_2}
\]
that $e^{t\Delta_{B^c}}$ can be regarded as 
the bounded operator from $L^{r_2}_{d\mu}$ to $L^{r_1}(B^c)$ with
\[
\|e^{t\Delta_{B^c}}f\|_{L^{r_1}(B^c)}
\leq 
\left(
\frac{C_{{\rm A},2,\frac{q}{p}}}{t^{1-\frac{p}{q}}(1+\log (1+t))}
\right)^{1-\theta}\|f\|_{L^{r_2}_{d\mu}}, 
\quad 
f\in L_{d\mu}^{r_2}. 
\]
Choosing $\theta=1-\frac{1}{p}$, 
we obtain the desired inequality. 
\end{proof}
\subsection*{Acknowedgements}

This work was supported by JSPS KAKENHI Grant Numbers 
20K14346, 22H00097 and 23K03174.


\end{document}